\documentclass[11pt]{article}
\usepackage{amsfonts,amsmath,amssymb,amsthm}
\usepackage{hyperref}
\usepackage{graphicx}
\usepackage{color}
\usepackage[affil-it]{authblk}
\usepackage{bbm}

\newcommand{\detail}[1]{\par\noi{\bf [Proof detail\ }{#1}
\hfill{\bf ]}\par\noi\hspace{-4pt}}
\renewcommand{\detail}[1]{}

\newcommand{\noi}{\noindent}

\newtheorem{theorem}{Theorem}[section]
\newtheorem{proposition}[theorem]{Proposition}
\newtheorem{corollary}[theorem]{Corollary}
\newtheorem{conjecture}[theorem]{Conjecture}
\newtheorem{lemma}[theorem]{Lemma}

\newtheorem{remark}[theorem]{Remark}

\newcommand{\et}{\end{theorem}}
\newcommand{\bl}{\begin{lemma}}
\newcommand{\el}{\end{lemma}}
\newcommand{\bp}{\begin{proposition}}
\newcommand{\ep}{\end{proposition}}
\newcommand{\bcor}{\begin{corollary}}
\newcommand{\ecor}{\end{corollary}}
\newcommand{\br}{\begin{remark}\rm}
\newcommand{\er}{\end{remark}}
\newcommand{\bcon}{\begin{conjecture}}
\newcommand{\econ}{\end{conjecture}}

\newcommand{\be}{\begin{equation}}
\newcommand{\ee}{\end{equation}}
\newcommand{\ba}{\begin{array}}
\newcommand{\ea}{\end{array}}
\newcommand{\bc}{\be\begin{array}{r@{\,}c@{\,}l}}
\newcommand{\ec}{\end{array}\ee}

\newcommand\cF{{\cal F}}

\newcommand{\eps}{\varepsilon}

\newcommand{\cA}{{\cal A}}
\newcommand{\Bi}{{\cal B}}
\newcommand{\Ci}{{\cal C}}

\newcommand{\cL}{{\cal L}}
\newcommand{\Mi}{{\cal M}}

\newcommand{\cY}{{\cal Y}}

\newcommand{\R}{{\mathbb R}}
\newcommand{\N}{{\mathbb N}}
\newcommand{\Z}{{\mathbb Z}}

\renewcommand{\P}{{\mathbb P}}
\newcommand{\E}{{\mathbb E}}

\newcommand{\ben}{\begin{enumerate}}
\newcommand{\een}{\end{enumerate}}

\newcommand{\beqn}{\begin{eqnarray}}
\newcommand{\eeqn}{\end{eqnarray}}
\newcommand{\beqnn}{\begin{eqnarray*}}
\newcommand{\eeqnn}{\end{eqnarray*}}

\newcommand\bx{{\vec x}}

\title{Recovering the Brownian Coalescent Point Process from the Kingman Coalescent by Conditional Sampling }

\author{Amaury Lambert, Emmanuel Schertzer}
\begin{document}

\maketitle

\begin{abstract}
We consider a continuous population whose dynamics is described by the standard stationary Fleming-Viot process, so that the genealogy of $n$ uniformly sampled individuals is distributed as the Kingman $n$-coalescent. In this note, we study some genealogical properties of this population when the sample is conditioned to fall entirely into a subpopulation with most recent common ancestor (MRCA) shorter than $\eps$.
First, using the comb representation of the total genealogy \cite{LUB16}, we show that the genealogy of the descendance of the MRCA of the sample on the timescale $\eps$ converges as $\eps\to 0$. The limit is the so-called Brownian coalescent point process (CPP) stopped at an independent Gamma random variable with parameter $n$, which can be seen as the genealogy at a large time of the total population of a rescaled critical birth-death process, biased by the $n$-th power of its size. Secondly, we show that in this limit the coalescence times of the $n$ sampled individuals are i.i.d. uniform random variables in $(0,1)$.
These results provide a coupling  between two standard models for the genealogy of a random exchangeable population: the Kingman coalescent and the Brownian CPP.

\end{abstract}

\section{Introduction}
In this paper, we seek to couple two well-known probabilistic objects both modeling the genealogy of an exchangeable population. The first object is the celebrated \emph{Kingman $n$-coalescent} \cite{K82}, which arises as the genealogy of a sample of $n$ co-existing individuals within a stationary population of large size $N$ constant through time, when time is measured in units of $N$ generations. The second object is the \emph{coalescent point process} introduced in \cite{P04}, which arises as the genealogy of the whole population at time $N$ of a critical birth--death process starting from size $N$ and with birth and death rates both equal to $1$, when time is also accelerated by $N$.

In the coalescent point process, the population is assumed to be endowed with a linear order consistent with the genealogy, in the sense that in a plane representation of this genealogy, lineages only intersect at internal nodes (common ancestors) --see Fig. \ref{cannings} or Fig. 7 in \cite{L08}. This order can also be obtained as the order inherited from a contour of the tree \cite{P04, L10, LP13}. The linear arrangement of coalescence times (that is, times to the MRCA --most recent common ancestor) between consecutive pairs of individuals ranked in the linear order converges to the concatenation of (a Poisson number with parameter 1 of) i.i.d. Poisson point processes with intensity measure $2dt\, x^{-2}dx$ killed at their first atom with second coordinate larger than 1. Each of these killed Poisson point processes encodes the genealogy of the descendance of an individual in a critical branching process conditioned on survival up to a large time. They will hereafter be called killed Brownian coalescent point process (killed Brownian CPP).\\

Both the standard Kingman coalescent and the killed Brownian CPP code for the genealogy of a large exchangeable population, but there are two features distinguishing them. First, the Kingman coalescent focuses on sparsely sampled individuals whereas the CPP deals with the whole population. Second, the Kingman coalescent is based on the assumption of a stationary population with constant size (total size constraint), whereas in the CPP the size of the population is only constant in expectation, and its foundation time is fixed (time constraint). 

Rather than defining a third object coupling the Kingman coalescent and the CPP, our aim is to show that one of the two is embedded in the other. Due to the first aforementioned distinctive feature (i.e., sparse sampling), one might think at first sight that the Kingman coalescent can be obtained by sparsely sampling the CPP. But in doing this, one would not get rid of the second distinctive feature, namely the time constraint.
The alternative view is the right one. In an exchangeable population with large constant size, the descendance of a small subpopulation is blind to the total size constraint and it is constant in expectation (see for example Theorem 1 in \cite{BL06}). Our goal is to prove a backward-in-time version of the last informal statement (`in a large stationary population with constant size, the genealogy of a subpopulation with recent MRCA is given by a CPP') and to derive some consequences of this fact.\\  


We start from the genealogy of a population with constant size in the stationary case directly with the continuous limiting object, the standard Fleming-Viot process \cite{FV79}. Actually, we will make use of an alternative description of the Fleming-Viot process, namely the flow of bridges introduced by Bertoin and Le Gall \cite{BL03}, in which the population is endowed with a linear order consistent with the genealogy (See Section \ref{sec:flow} and Fig. \ref{cannings}). We will call \emph{Kingman comb} the list of coalescence times of pairs of `consecutive individuals' in this linearly arranged continuous population, as defined in Section \ref{sec:comb}. In Section \ref{sec:comb},  we show that the properly rescaled Kingman comb converges to the Brownian CPP  (see Proposition \ref{cvc2}).

In the remainder of the paper, we investigate the genealogy of $n$ individuals sampled in such a way that their MRCA lies at a depth smaller than $\eps$ with $\eps\to0$. Note that this conditioning can be implemented in two distinct ways:
\begin{enumerate}
\item[(i)]({\it Quenched conditional sampling}) Conditional on the flow of bridges, sample $n$ individuals such that their MRCA lies at a depth smaller than $\eps$ and then average over every realization of the flow. In biology, such conditioning could arise by sampling on purpose individuals that share close phenotypic characteristics or dwell in neighboring habitats.
\item[(ii)]({\it Averaged conditional sampling}) Directly condition the $n$-Kingman coalescent to have its MRCA lie at a depth smaller than $\eps$. In contrast with (i), where the conditioning is solely enforced at the sampling level, 
the conditioning in (ii) could be due either to anomalous sampling (as in (i)) or to an abnormally shallow MRCA of the {\it entire} population. 
\end{enumerate} 
In Section \ref{c-s},  we focus on case (i), where we consider the entire family that shares a common ancestor with the $n$ sampled individuals. We show that for small $\eps$, the genealogy of this family is given by a CPP killed at an independent Gamma random variable (see Theorem \ref{cvc}). Informally, this amounts to saying that the genealogy of the family of the sample is the rescaled genealogy of a {\it $n$-size-biased} critical birth-death process (i.e., biased by the $n$-th power of its size) conditioned on survival up to a large time (see Remark \ref{rem-cr}). 

In Section \ref{c-s-2}, we use this result to prove that the genealogy of the $n$ conditionally sampled individuals enjoy a nice i.i.d. structure, namely that their (properly rescaled) $n-1$ coalescence times are i.i.d. uniform random variables (Theorem \ref{teo1}). In Section \ref{c-s-3}, we turn to case (ii) where we show (Theorem \ref{petit-calcul}) that the genealogy of the $n$ individuals is also described (asymptotically)
in terms of i.i.d. uniform random variables, thus showing that the conditionings (i) and (ii) become indistinguishable as $\eps$ goes $0$. 

Finally, we briefly mention a natural conjecture arising from the previous results.  Because (a) the genealogy of the $n$ individuals in (i) and (ii) coincide asymptotically as $\eps\to 0$, and because (b) the entire family sharing a common ancestor with the $n$ individuals in case (i) is described in terms of a `$n$-size-biased'  killed Brownian CPP (see again Theorem \ref{cvc}), it is natural to conjecture that Theorem \ref{cvc} also holds in case (ii).

\section{Preliminaries: Flows of Bridges and Combs}
\label{sec:flow}

\subsection{Discrete Bridges} 

Flows of bridges have been introduced by Bertoin \& Le Gall in \cite{BL03}. In order to motivate their construction, 
let us consider a general discrete time Cannings \cite{C75} model as follows.
\begin{itemize}
\item[(1)] At each generation the size of the population
is fixed and equal to $N$;
\item[(2)] Individuals at generation $r$ are labelled from $1$ to $N$ and we denote by  $(\nu^1_r,\ldots,\nu^N_r)$
the vector of offspring numbers;
\item[(3)] This labeling of individuals is consistent with the genealogy (cf. Introduction and Fig.~\ref{cannings});
\item[(4)] The vectors $((\nu^1_r,\cdots,\nu^N_r); r\in \Z)$ are i.i.d. exchangeable vectors.
\end{itemize}
Recall from the Introduction that a labeling consistent with the genealogy is such that lines of descent only cross at internal nodes (See Fig. \ref{cannings}). Rigorously, this amounts to enforcing the condition that if $i<j$, the label of an offspring of individual $i$ is always smaller
than the label of an offspring of individual $j$.

Now for any $x \in \ \{0,\cdots,N\}$ and $r\in\Z$, define $B_{r,r+1}(x)$ as the number of individuals at generation $r+1$ descending from the subpopulation with labels smaller than or equal to $x$ at generation $r$. Thanks to Assumption (2), we have
$$
B_{r,r+1}(x) \ :=  \ \sum_{k\leq x }  \ \nu_{r}^k.
$$
Thanks to Assumptions (1) and (4), the maps $(B_{r,r+1};r\in\Z)$ are i.i.d. and each $B_{r,r+1}$ is a \emph{discrete bridge}, that is a non-decreasing function from $ \{0,\ldots,N\}$ onto itself with exchangeable increments. For any $m<n$, define 
$$
B_{m,n}  \ := \  B_{n-1,n} \circ \cdots \circ \ B_{m,m+1}.
$$
Thanks to Assumption (3), $B_{m,n}(x)$ is the number of individuals at generation $n$ descending from the subpopulation with labels smaller than $x$ at generation $m$. The bridge property is stable under composition and furthermore,
$\{B_{m,n}\}_{m<n}$ satisfies the so-called cocycle property
$$
B_{k,n} =  B_{m,n} \circ B_{k,m} \qquad k < m <n.
$$
The collection $\{B_{m,n}\}_{m<n}$ is called a discrete \emph{flow} of discrete bridges. Clearly, the increments of the flow are stationary and independent.\\
\\
Let us define the  inverse flow $\phi_{m,n}$ by
\beqn
\phi_{m,n}(x) = \inf\{ y \in \{1,\ldots, N\}:  B_{m,n}(y) \geq x \} \qquad m<n,  \ x\in\{1,\ldots,N\}. \label{inv} 
\eeqn
The  bridge property implies that $\{\phi_{m,n}\}_{m<n}$ defines a backward coalescing flow, in the sense that $\phi_{k,n} = \phi_{k,m} \circ \phi_{m,n}$ for $k < m <n$ and the orbits $\{\phi_{m,n}(x)\}_{m\leq n}$ coalesce upon meeting each other as $m$ decreases. More specifically, if $\phi_{m,n}(x) =\phi_{m,k}(x')$, then $\phi_{m',n}(x) =\phi_{m',k}(x')$ for all $m'<m$. For every $x\in\{1,\cdots,N\}$, we can then see the orbit 
\beqnn
\N & \rightarrow & \{1,\ldots,N\} \\
k                   & \mapsto & \phi_{n-k,n}(x)
\eeqnn
as the \emph{ancestral lineage} of individual $x$ of generation $n$. (See Fig. \ref{cannings}.)

\begin{figure}[h]
\centering
\includegraphics[scale=.3]{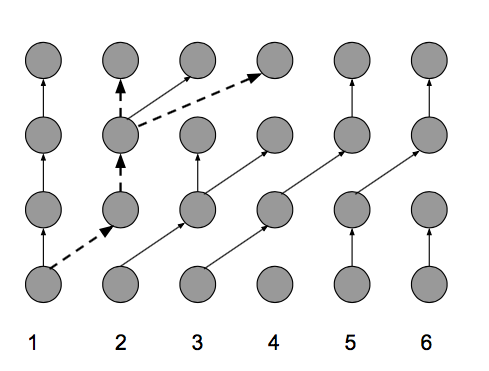}
\caption{A realization of the Cannings model at times $t=0,1,2,3$ flowing upwards. Dashed lines indicate 
the ancestral lineages of individuals $2$ and $4$ at generation $3$.
At each generation, individuals are labelled from 1 to 6 so that
ancestral lines do not cross.}
\label{cannings}
\end{figure}

\subsection{The Standard Fleming-Viot Flow of Bridges} 

In a similar way to the discrete flow of bridges, we now define the continuous flow of bridges as done in \cite{BL03}. Now a \emph{bridge} is a non-decreasing function $B$ from $[0,1]$ onto itself with exchangeable increments such that $B(0)=0$ and $B(1) =1$. A stochastic \emph{flow of bridges} is a family $\{B_{s,t}\}_{-\infty<s<t<\infty}$ of bridges satisfying the following properties.
\begin{itemize}
\item[(1)]
 Co-cycle property. For any fixed $r<s<t$, $B_{r,t} =  B_{s,t} \circ B_{r,s}$ a.s.;
 \item[(2)] Independent and stationary increments; 
 \item[(3)] No fixed time discontinuity. For any fixed time $s$, $\lim_{t\downarrow s} B_{s,t} = \mbox{Id}$ (uniformly) in probability. 
\end{itemize}
We think of a stochastic flow as the dynamics of a stationary, continuous population with constant size equal to $1$. The genealogy of the population alive at time $t$ is encoded by the backward coalescing flow $\{\phi_{s,t}\}_{s<t}$ defined analogously to \eqref{inv} by  
\beqn
\phi_{s,t}(x) = \inf\{ y \in [0,1]:  B_{s,t}(y) > x \} \qquad s<t,  \ x\in[0,1]. \label{invcont} 
\eeqn
Since the increments of a flow are stationary and independent, a flow is uniquely characterized by its one-dimensional marginal $B_{0,t}$. 
In this paper, we will specifically consider the so-called \emph{standard Fleming-Viot} (FV) flow of bridges whose one-dimensional marginal  is equal to
$$
B_{0,t}(x) =
\sum_{i=1}^{N_t} \mathbbm{1}_{[0,x]}(U_i)\  \beta_i
$$
where 
\begin{itemize}
\item[(1)] $N_t$ is distributed as the value at time $t$ of a pure-death process going from $k$ to $k-1$ at rate $k(k-1)/2$ and started at $\infty$; 
\item[(2)] Conditional on $N_t=n$, the random vector  $\{\beta_i\}_{i=1}^{n}$ is independent of the $(U_i)$ and follows the Dirichlet distribution
with parameter $(1,\ldots,1)$;
\item[(3)] $\{U_i\}_{i=1}^{\infty}$ is a sequence of i.i.d. uniform random variables independent of $(N_t; \{\beta_i\}_{i=1}^{N_t})$. 
\end{itemize}
In the same spirit as \cite{MS01}, this flow should arise as the scaling limit of discrete bridges induced by any Cannings model with enough control on the tail of the offspring distribution (in particular the Wright-Fisher model).

\subsection{Combs and Coalescent Point Processes}

In the next section, we will be interested in the backward coalescing flow associated with the FV flow of bridges. The precise trajectory of the ancestral lineage of a given individual  -- i.e., the successive labels in (0,1) of the ancestors of this individual -- is in most applications irrelevant. In contrast, one would like to extract from the coalescing flow the pure genealogical information contained in this object. 

To do this we follow \cite{LUB16}  and define a \emph{comb} as a function $f:[0,+\infty)\to [0,+\infty)$, such that for any $\eps>0$, $\{f\ge \eps\}$ is finite. Then $d_f(x,y):= \sup_{z\in(x\wedge y, x\vee y)} f(z)$ defines an ultrametric distance on $\{f=0\}$ called the comb metric (modulo the identification of points at distance 0 if $f$ is zero on one or several open intervals). Let $\Omega$ be the space of combs and consider the mapping
\begin{eqnarray*}
F: \Omega & \to & {\cal M}  \\
f             & \mapsto & \sum_{x: f(x)>0} \delta_{x,f(x)}
\end{eqnarray*}
where ${\cal M}$ denotes the space of point measures on $(0,\infty)^2$. We assume that $\cal M$ is endowed with the topology induced by test-functions $\varphi:(0,\infty)^2\to \R$ which are continuous and bounded and for which there is $M>0$ and $\eps>0$ such that $\varphi=0$ outside $(0,M)\times (\eps,\infty)$. We equip $\Omega$ with the $\sigma$-field generated by $F$ when $\cal M$ is equipped with its Borel $\sigma$-field.


Let $(C_n)$ be a sequence of random combs such that for every $A\subseteq (0,\infty)^2$ with zero Lebesgue measure, $F(C_n)(A)=0$ a.s. for every $n$. Let $C$ be a random comb with the same property. We will use repeatedly the fact that $(C_n)$ converges weakly in law to $C$ iff for any $x_1< \cdots< x_k$, the random vectors $(d_{C_n} (x_i,x_{i+1}); 1\le i\le k-1)$ converge in law to the random vector $(d_{C} (x_i,x_{i+1}); 1\le i\le k-1)$. This can be seen thanks to the Portmanteau theorem and thanks to the following equality between events
$$
\{d_{C} (x_i,x_{i+1})<y_i, \forall i\}\ =\ \{F(C)\left(A\right)=0\},
$$
where $A$ is the following subset of $(0,\infty)^2$ 
$$
A=\cup_{i=1}^{k-1}(x_i, x_{i+1}) \times [y_i,\infty),
$$
which is such that $F(C)(\partial A)=0$ a.s.

For any comb $f$ and $t>0$, we will set
$$
k_t(f):=f\mathbbm{1}_{[0,t)} 
$$
and call it the comb $f$ \emph{killed at} $t$.
For any $x>0$, we define 
$$
l_x(f):=\inf\{y\ge 0: f(y) >x\}.
$$
For any $a>0$, we define the scaling operator $S_a$ by
$$
S_a(f)(x):= a^{-1} f\left(a^{-1}x\right).
$$
In particular, $S_a\circ k_t = k_{t/a}\circ S_a$ and $S_a\circ k_{l_x} = k_{l_{x/a}}\circ S_a$.

For any $\sigma$-finite measure $\nu$ on $(0,\infty)$, the \emph{coalescent point process} (CPP) with intensity measure $\nu$ is the random comb $C$ such that $F(C)$ is a Poisson point process with  intensity measure $dt\otimes \nu(dx)$ . 
The coalescent point process associated to 
$$
\nu(dx) \ = \ \frac{2}{x^2} dx,
$$
will hereafter be denoted $\Ci$ and called the \emph{Brownian CPP}. 
We will also define $\bar \Ci:= k_{l_1}(\Ci)$ referred to as the  \emph{killed Brownian CPP}.
As mentioned in the introduction, $\bar \Ci$ can be thought of as the genealogy rescaled by $N$ of the descendance of an individual by a critical branching process conditioned on survival up to a large time $N$.

\begin{remark}
Based for example on \cite{LG05}, it is known that the reflected Brownian motion codes in a certain appropriate sense for a rescaled critical branching forest. For a forest coded by a non-negative function $h$, the coalescence times of the part of the tree lying at distance $d$ from the root, are the depths of the excursions of $h$ away from $d$. This explains why the measure $\nu$ is (up to a multiplicative constant) the It\^o measure of Brownian excursion depths \cite{P04}. 
\end{remark}

\section{The Kingman Comb at Small Scale} 
\label{sec:comb}

\subsection{The Kingman Comb}

Let $\phi$ be the backward coalescing flow defined in \eqref{invcont} from the standard Fleming-Viot flow of bridges. For every $x<y\in(0,1)$, the coalescence time of $x$ and $y$ is given by
$$
\inf\{t>0:  \phi_{-t,0}(x) = \phi_{-t,0}(y)  \}. 
$$
The next statement shows that this genealogical structure can be represented as a comb that we call the Kingman comb (see also \cite{K82, LUB16} for other treatments of the Kingman comb). 

\begin{proposition}
\label{prop:Kingman-comb}
There is a sequence $(V_j)$ of i.i.d. uniform random variables and an independent sequence $(T_j)$ where $T_j=\sum_{k\ge j+1} e_k$, for $e_k$ independent exponential r.v. with parameter $k(k-1)/2$, such that
$$
\inf\{t>0: \phi_{-t,0}(x) = \phi_{-t,0}(y)  \} \ =\ d_C(x,y) \qquad\mbox{for every  }x<y \ \mbox{in $\{C=0\}$   }\mbox{ a.s.}
$$
where
\begin{equation}
\label{eqn:def-comb}
C = \sum_{j\ge 1} T_j \mathbbm{1}_{\{V_j\}}.
\end{equation}
The function $C$ is a comb a.s., and the distance $d_C$ is the comb metric associated to $C$. Thereafter, we will call this random comb the \emph{Kingman comb}.  
\end{proposition}
\begin{proof}
The fact that $C$ is a comb is straightforward.  Instead, we focus on the first part of Proposition \ref{prop:Kingman-comb}.
Let $(W_n)$ be an independent sequence of i.i.d. uniform r.v. and for any $t>0$ define the equivalence relation $\sim_t$ in $\N$ by
$$
k\sim_t n \Longleftrightarrow \phi_{-t,0}(W_k) = \phi_{-t,0}(W_n).
$$
We denote by $\Pi_t$ the partition of $\N$ induced by $\sim_t$. It is known from \cite{BL03} that $(\Pi_t;t\ge 0)$ (has a c\`adl\`ag modification which) is distributed as the standard Kingman coalescent. In particular, the number of blocks $N_t$ of $\Pi_t$ is a non-increasing process started at $\infty$ which jumps from $k$ to $k-1$ at rate $k(k-1)/2$, so that the intersection $R_t$ of $(0,1)$ with the range of $B_{-t,0}$ is finite (with cardinal $N_t-1$) and non-increasing. Let $J_1>J_2>\cdots$ denote the jump times of $(N_t)$ labelled in decreasing order and for any $i\ge 1$ let $V_i$ be the unique element of $(0,1)$ such that  
$$
R_{J_{i}}=R_{J_{i-1}}\cup \{V_i\}
$$
We also know that $B_{-t,0}$ can be written as  
$$
B_{-t,0}(x) =
\sum_{i=1}^{N_t} \mathbbm{1}_{[0,x]}(U_i)\  \beta_i
$$
where conditional on $N_t =n$, the vector $(\beta_i)$ follows the Dirichlet distribution with parameter $(1,\ldots,1)$. This means that for all $n$, the vector of diameters of the connected components of $(0,1)\setminus\{V_1,\ldots, V_{n-1}\}$ follows the Dirichlet distribution with parameter $(1,\ldots,1)$. Standard arguments imply that the $(V_i)$ form a sequence of i.i.d. uniform r.v. independent of the sequence $(W_n)$ (because they depend deterministically upon the flow of bridges).

Next define $e_k:= J_{k-1}-J_k$ and $T_j=\sum_{k\ge j+1} e_k$. We have already mentioned that $e_k$ is exponentially distributed with parameter $k(k-1)/2$ and it is well-known that the $(e_k)$ are independent of the asymptotic frequencies of the blocks of $(\Pi_t)$, and so are independent of the $(V_i)$. Defining $C$ as in \eqref{eqn:def-comb}, we easily see that a.s. for any $k,n$, for any $t\in [J_i, J_{i-1})$ ($i\ge 1$, $J_0:=+\infty$),
$$
k\sim_t n \Longleftrightarrow \forall j\in\{1,\ldots, i-1\}, \ V_{j}\not\in (W_k\wedge W_n, W_k\vee W_n)\Longleftrightarrow \sup_{(W_k\wedge W_n, W_k\vee W_n)}C <t.
$$
Taking the union of $i$, this can be expressed as follows. Almost surely for any $k,n$, for any $t$ which is not a jump time of $N$, 
$$
\phi_{-t,0}(W_k) = \phi_{-t,0}(W_n) \Longleftrightarrow \sup_{(W_k\wedge W_n, W_k\vee W_n)}C <t.
$$
By density of the $(W_n)$, this implies that a.s. for all $x<y\in(0,1)\setminus \{U_i\}$, for all $t$ which is not a jump time of $N$, 
$$
\phi_{-t,0}(x) = \phi_{-t,0}(y) \Longleftrightarrow \sup_{(x,y)}C <t,
$$
which terminates the proof.
\end{proof}


\begin{remark}
Proposition \ref{prop:Kingman-comb}  is interesting in its own right. It provides a natural interpretation of the r.v. $V_i$
appearing in the definition of the Kingman comb.  
Thinking of the flow of bridges as describing the dynamics of a population of constant size $1$, the $T_i$'s
indicate the dates of branching events giving rise to lineages surviving both until the present. 
For a given value of $T_i$, the two resulting extant subpopulations can be identified with the interval $(V_i^{-},V_i)$  and $(V_i,V_i^+)$ where
\beqnn
V_i^+& = & \inf\{V_j>V_i\ : \  T_j >T_i\}, \\
V_i^-& = & \sup\{V_j <V_i\ : \ T_j >T_i\}, 
\eeqnn
with the convention $\inf\{\emptyset\}=0$ and $\sup\{\emptyset\}=1$. To conclude, not only does the comb 
encapsulate the time and the linear ordering of splitting events that are relevant to the present (the $T_i$'s), but it also retains the size of the sub-populations 
arising from those splitting events (the $V_i$'s).
\end{remark}

\begin{remark}
The comb at time $t=0$ is generated from the sequence $(\phi_{-t,0}; t\geq0)$. Analogously,
for every time $s\in\R$,
one can define a comb $C_s$ from the sequence $(\phi_{t,s}; t\leq s)$. $(C_s; s\in \R)$ naturally defines a stationary 
stochastic process that
will be the subject of future work. We expect that this `evolving Kingman comb' will shed new light on the evolving Kingman coalescent
 studied for example in \cite{PW06} and \cite{PWW11}.
\end{remark}

\subsection{Convergence to the Brownian CPP}
The next proposition relates the Kingman comb at small scale with the Brownian CPP.

\begin{proposition}\label{cvc2} 
The following convergence 
$
S_\eps(C) \Longrightarrow \Ci
$
holds weakly in law as $\eps\to 0$. 
\end{proposition}

\begin{proof}

Let $\eps>0$. For every $t_0>0$ and every bounded open interval $I\subseteq (0,\infty)$, define
$$
F_\eps^{I,t_0} \:= \ \{(u_\eps,s_\eps) \in (t_0,\infty)\times I \ :\  S_\eps(C)(u_\eps) = s_\eps \}.
$$
In order to prove Proposition \ref{cvc2}, it is enough to show that
the random set $F_\eps^{I,t_0}$ converges weakly in the vague topology to a PPP on $(0,\infty)^2$ with intensity measure
$$
2 \  1_{x\in I} dx   \ 1_{t\geq t_0} \frac{dt}{t^2}
$$
as $\eps\to 0$. For every $t>0$, define $N_t := \#  \{j \ : \ T_j> t\}$ corresponding to the block counting process of the Kingman coalescent. By definition,  the set $F_\eps^{I,t_0}$
coincides with the set of points of the form $\frac{1}{\eps}(V_i,T_i)$ with $(V_i,T_i)$ belonging to the Kingman comb (as defined in Proposition \ref{prop:Kingman-comb}) and such that 
$$
V_i \in \eps I \ \ \mbox{and} \ \ i\in\{1,\cdots, N_{\eps t_0}\}.
$$
Next, let $c>0$ and let us compare the previous set with 
the set $\bar F_\eps^{I,c}$
consisting of every point of the form $\frac{1}{\eps}(V_i,T_i)$, with $(V_i,T_i)$ again belonging to the Kingman comb, but such that 
$$
V_i \in \eps I \ \ \mbox{and} \ \ i\in\{1,\cdots, [\frac{c}{\eps}]\}.
$$
We claim that $\bar F_\eps^{I,c}$ converges to a PPP with intensity measure
\be\label{eq:i-m-c}
2 \  1_{x\in I} dx   \ 1_{t\geq \frac{2}{c}} \frac{dt}{t^2}.
\ee
Before justifying the claim, let us briefly explain how this entails Proposition \ref{cvc2}. 
On the one hand, for any $c_1<c_2$
$$
\P\left( \   \bar F_\eps^{I,c_1}\  \leq \ F_\eps^{I,t_0}  \ \leq \  \bar F_\eps^{I,c_2} \right) \geq \P\left( N_{\eps t_0} \in\{[\frac{c_1}{\eps}], \cdots, [\frac{c_2}{\eps}]\} \right).
$$ 
It is well known that the renormalized block counting process $\eps t_0  N_{\eps t_0} $ converges to $2$ in probability as $\eps \downarrow 0$. Thus, for any $\delta>0$, 
taking $c_1 = \frac{2-\delta}{t_0}$ and $c_2 = \frac{2+\delta}{t_0}$ in the latter inequality yields
$$
\P\left( \   \bar F_\eps^{I,c_1}\  \leq \ F_\eps^{I,t_0}  \ \leq \  \bar F_\eps^{I,c_2} \right)  \ \to \ 1 \ \ \mbox{as $\eps\to 0$}. 
$$ 
Finally, since this holds for every $\delta>0$,  Proposition \ref{cvc2} easily follows by letting $\delta\to 0$ (assuming that $\bar F_\eps^{I,c}$ converges to a PPP with the intensity measure provided in  (\ref{eq:i-m-c})). \\

It remains to justify the convergence of $(\bar F_\eps^{I,c})$.
Recall from Proposition \ref{prop:Kingman-comb} that the r.v. $V_1,V_2, \cdots$ in the Kingman comb are i.i.d. uniform r.v. on $[0,1]$. As a consequence,
the cardinality $\# \bar F_\eps^{I,c}$ of  $\bar F_\eps^{I,c}$ is distributed as a Binomial r.v. with parameters $([\frac{c}{\eps}], \eps |I|)$,
where $|I|$ refers to the Lebesgue measure of $I$. Standard arguments yield that 
$\# \bar F_\eps^{I,c}$ converges to a Poisson random variable with parameter ${c |I|}$.  \\


Next, let $(u^\eps, \sigma^\eps)$ denote the atoms of $(\bar F_\eps^{I,c})$. Fix $m\in \N^*$ and for every $\eps>0$, let us condition on the event
$\{\# \bar F_\eps^{I,c} = m\}$. (Note that the law of $(T_i; i\geq 0)$ is not affected by the conditioning). This defines a sequence of r.v. $(i_1^\eps,\cdots,i_m^\eps)$ 
of distinct integers in
$\{1,\cdots,[\frac{c}{\eps}]\}$ such that 
$$
(u^\eps,\sigma^\eps) \ = \ \frac{1}{\eps}\left( (V_{i_k^\eps}, T_{i_k^\eps}); \ k\leq m \right).
$$
The statistical description of the comb provided in Proposition \ref{prop:Kingman-comb} implies that:
\begin{enumerate}
\item[(i)] $(V_{i_k^\eps}; k\ \leq m)$ are i.i.d. uniform random variables on $\eps I$, independent of the sequence $\left(  T_{i_k^\eps}; \  k\leq m \right)$.
\item[(ii)] $(i_1^\eps,\cdots,i_m^\eps)$ is a uniform sample  of size $m$ (with no replacement) of $\{1,\cdots,[\frac{c}{\eps}]\}$ independent
of the $T_i$'s. 
\end{enumerate}
From (ii), we get:
\be\label{eq:cv-i}
\frac{\eps}{c}(i_1^\eps,\cdots,i_m^\eps) \ \to \ (U_1,\cdots,U_m) \  \ \mbox{in law},
\ee
where $U_1,\cdots,U_m$ are i.i.d. uniform r.v. on $[0,1]$. We now prove that 
\be\label{eq:cv-T}
\forall k\leq m, \ \ i_k^\eps T_{i_k^\eps} \ \to  \ 2 \ \ \mbox{in probability.}
\ee
In order to see that, we first note that  the $i_k T_{i_k^\eps}$'s are exchangeable and 
to ease the notation, we remove the $k$-subscript until further notice.
For any random variable $X\in L^1$, write $\E_{m,\eps}(X):=\E(X \ | \ \# \bar F_\eps^{I,c} = m )$. Then, using (ii) above: 
\beqn
\E_{m,\eps}(  \ i^\eps T_{i^\eps} \ | \ i_\eps) 
& = & i_\eps \sum_{k\geq i_\eps+1} \frac{2}{k(k-1)}  \  = \ \ \frac{1}{i_\eps} \sum_{k\geq i_\eps+1} \frac{2i_\eps ^2}{k(k-1)} \nonumber \\ 
&= &   \ 2 \int_1^\infty \frac{dy}{y^2} + R_1(i_\eps) =  2 + R_1(i_\eps), \label{eq:cond-ex}
\eeqn
where $R_1$ is a function such that $|R_1(x)|\leq K/x$ for some constant K. Averaging of $i^\eps$, we get:
\be\label{eq:conv-E}
\E_{m,\eps}(  \ i^\eps T_{i^\eps} \ ) \ = \ 2  \ + \ \sum_{k=1}^{[\frac{c}{\eps}]} R_1(k) \ \P(i^\eps=k) \ \to \ 2,  
\ee
using the fact that $\P(i^\eps=k) = 1/[\frac{c}{\eps}]$ and the previous bound on $R_1$. Next,
$$
\mbox{Var}_{m,\eps}(i^\eps T_{i^\eps}) \ = \ \E_{m,\eps}\left(\mbox{Var}_{m,\eps}(i^\eps T_{i^\eps} \ | \ i^\eps)\right) \ + \ \mbox{Var}_{m,\eps}\left(\E_{m,\eps}(i^\eps T_{i^\eps} \ | \ i^\eps)\right).
$$
First, it is not hard to see from (\ref{eq:cond-ex}) that $\E(i^\eps T_{i^\eps} \ | \ i^\eps)$ converges to $2$ in $L^2$, and thus,
the second term on the RHS of the latter equality vanishes as $\eps\to 0$. Let us now deal with the first term
\beqnn
\mbox{Var}_{m,\eps}(i^\eps T_{i^\eps} \ | \ i^\eps) & = & (i^{\eps})^2 \sum_{k\geq i^\eps+1} \left(\frac{2}{k(k-1)}\right)^2  \  = \ \ \frac{1}{i^\eps} 
\left(\frac{1}{i^\eps} \sum_{k\geq i^\eps+1} \left(\frac{2( i^\eps)^2}{k(k-1)}\right)^2\right) \ \\
& = &   \frac{1}{i^\eps} \int_1^\infty \frac{ 4 dy}{y^4}   \ + \ R_2(i_\eps),
\eeqnn
where $R_2$ is a function such that $|R_2(x)|\leq K'/x^2$ for some constant $K'$. Reasoning as in (\ref{eq:conv-E}), 
this shows that the expectation of the RHS of the last equality goes to $0$ as $\eps\to0$. Altogether, this implies that
$$
\mbox{Var}_{m,\eps}(i^\eps T_{i^\eps})  \ \to 0 \ \ \mbox{as $\eps\to0$}.
$$
Together with (\ref{eq:conv-E}), this completes the proof of (\ref{eq:cv-T}). Finally, combining  (\ref{eq:cv-i}) with  (\ref{eq:cv-T})  yields the following convergence in law as $\eps\downarrow 0$
\be\label{eq:cv-T2}
\left(\frac{1}{\eps} T_{i^\eps_1},\cdots, \frac{1}{\eps} T_{i^\eps_m}\right) \ \to \ \left(\frac{2}{c U_1},\cdots,\frac{2}{ c U_m}\right).
\ee 
Let us gather the previous arguments. From the convergence of $\# \bar F_\eps^{I,c}$ and (\ref{eq:cv-T2}), it is not hard to deduce that
$(\bar F_\eps^{I,c})$ is tight. 
Further,  we showed that any sub-sequential limit  $\cF^{I,c}$
must satisfy the following properties:
\begin{enumerate}
\item $\# \cF^{I,c}$ is distributed as a Poisson r.v. with parameter ${c|I|}$. 
\item  Let $\{(u,\sigma)\}$ denote the atoms of $\cF^{I,c}$.
Conditional on $\# \cF^{I,c}$: 
\begin{enumerate}
\item[(i)] $(\sigma)$ and $(u)$ are independent.
\item[(ii)] $(u)$ is a sequence of uniform r.v. on $I$.
\item[(iii)] $(\sigma)$ are i.i.d. r.v. with density 
$$
\frac{d}{dt}\P\left(\frac{2}{cU_1}\leq t\right) \ = \  \frac{d}{dt}\left(1 - \frac{2}{ct}\right)1_{t>2/c}   \ = \ 1_{t>2/c} \frac{2}{c t^2}.
$$
\end{enumerate}
\end{enumerate}
Now these two properties uniquely characterize the law of a PPP with intensity measure
 $$
2 \ 1_{x\in I} dx   \ \ 1_{t > \frac{2}{c}} \frac{dt}{t^2}. 
 $$ 
 This completes the proof of the convergence of $(\bar F_\eps^{I,c})$ and the proof of 
Proposition \ref{cvc2}.
\end{proof}

\subsection{Alternative Proof to Proposition \ref{cvc2}}

The proof of Proposition \ref{cvc2} is based on the statistical description of the Kingman comb given in Proposition \ref{prop:Kingman-comb}. 
Here, we sketch an alternative proof of Proposition \ref{cvc2} relying on totally different techniques, namely Ray-Knight Theorem and a result of Bertoin and Le Gall \cite{BL05} that describes the trajectories of the ancestral lineages in the FV flow of bridges. 

Let $x_0< \cdots < x_n$ and assume that $\eps$ is small enough such that $\eps x_n < 1$.
Define 
$$
\bar Y_i^\eps(t) \ := \ \frac{1}{\eps}\phi_{-\eps t,0}(\eps x_i), \ i\in\{0,\ldots,n\}
$$
By Proposition \ref{prop:Kingman-comb},  $\frac{1}{\eps} d_C(\eps x_{i-1}, \eps x_i)$ 
coincides with the 
hitting time at $0$ of the process $\bar Y_i^\eps(t) - \bar Y_{i-1}^\eps(t)$.
Thus, it remains to show that this vector of hitting times converges in law to $\left(\max_{[x_{i-1}, x_i]} {\cal C}, \ i=1,\ldots,n\right)$, or equivalently, that 
the components of the vector are asymptotically independent and distributed as $\max_{[0, x_i-x_{i-1}]} {\cal C}$ respectively.

Following Theorem 6 in \cite{BL05}, 
the $n$-point motion $(\bar Y_i^\eps, \  i=0,\ldots,n)$ is a coalescing diffusion whose generator is given by ${\cal A}^\eps$:
\begin{equation}
{\cal A}^\eps g\ = \ \frac{1}{2} \sum_{i,j=0}^n y_{i\wedge j} (1-\eps y_{i\vee j}) \frac{\partial g}{\partial y_i \partial y_j} \ + \  \sum_{i=1}^n (\frac{1}{2}-\eps y_i) \frac{\partial g}{\partial y_i},
\end{equation}
for every $g\in C_0^\infty(\R^{n+1})$.
Using standard arguments, one can prove that
$$
\left(\frac{1}{\eps} d_C(\eps x_{i-1}, \eps x_i), \ i=1,\ldots,n \right) \ \Longrightarrow \ \left(\ \inf\{ t \ : \ (\cY_i(t) - \cY_{i-1}(t)) = 0 \},\  i=1,\ldots,n\right).
$$
where $\{\cY_{i}\}_{i=0}^n$ is the diffusion with generator $\cA_0$ and initial condition $(x_0,\ldots,x_n)$.
The $n$-dimensional process $\{\cY_{i}-\cY_{i-1}\}_{i=1}^{n}$ can 
be rewritten as
 $M Y^0$ where $M$ is the $n\times (n+1)$ matrix:
 $$
 M \ = \ \left( \begin{array}{ccccc}  
 -1 & 1 & 0 & 0 & \cdots \\ 
 0  & -1 & 1 & 0 & \cdots \\
  0  & 0 & -1 & 1 & \cdots \\
 \vdots  & \vdots & \vdots & \ddots & \ddots \\ 
 \end{array} \right).
 $$
Thus, the generator of $\{\cY_{i}-\cY_{i-1}\}_{i=1}^{n}$ is given by
 $
 \sum_{i=1}^n d_i \frac{\partial }{\partial u_i} \ + \  \sum_{i,j=1}^n c_{i,j} \frac{\partial^2 }{\partial u_i\partial u_j} 
 $
where
\beqnn
d & = &  M \ \left(\begin{array}{c} \frac{1}{2} \\ \vdots \\ \frac{1}{2} \end{array} \right) \ = \ \ \left(\begin{array}{c} 0 \\ \vdots \\ 0 \end{array} \right) \\
c  & = & \frac{1}{2} M \ \left[  \ y_{i\wedge j} \ \right]_{i,j\leq n} \ ^t M.
\eeqnn
One can readily check that $c$ is diagonal with the diagonal terms given by $\frac{1}{2} \left(y_1-y_0,\cdots,y_n-y_{n-1}\right)$.
 Thus, the processes
 $(\cY_{i}-\cY_{i-1})$'s are distributed as independent Feller diffusions (i.e., satisfying  Eq (\ref{dif}) below) with respective initial condition $(x_i-x_{i-1})$.
The convergence of $S_\eps(C)$ to ${\cal C}$ is then a direct consequence of Lemma \ref{id} below.

\begin{lemma}\label{id}

Let $w$ be a standard Brownian motion. For every $x\in\R$, let $\sigma_x$ be the hitting time of $0$ by the diffusion
\begin{equation}\label{dif}
d z(t) = \sqrt{z(t)} d w(t),  \ \ z(0)= x.
\end{equation}
Then the following identity holds in distribution:
\begin{equation}\label{jhf}
\sigma_x \ = \ \sup_{[0,x]} {\cal C}.
\end{equation}
\end{lemma}

\begin{proof}
Let $x\geq0$. Define the local time
at $x$:
\begin{eqnarray*}
L^x(t) \ := \ \lim_{\eps\to 0} \frac{1}{2\eps} \ \int_0^t 1_{w(s) \in[x-\eps,x+\eps]} ds, \ \ {\cal L}^x(t) \ := \ \lim_{\eps\to 0} \frac{1}{2\eps} \ \int_0^t 1_{|w(s)| \in[x-\eps,x+\eps]} ds.
\end{eqnarray*}
and $\tau_y:=\inf\{u \ : \ L^0(u)>y\}$ to be the inverse local time at $0$.  
In order to prove Lemma \ref{id}, we will show that the RHS and LHS of (\ref{jhf}) are identical in law to $ \max_{s\in[0,\tau_{2x}]} |w(s)|$.

We start by showing that the indentity holds for the LHS  of (\ref{jhf}). Define $y(t)= 4 z(t)$. Then it is straightforward to check that $y$ is a standard $0$-dimensional squared Bessel
process (i.e., $d y(t) = \sqrt{2y(t)} d w(t)$) with initial condition $4x$. Furthermore, the hitting time 
at $0$ of $y$ coincides with the one of $z$. We now construct the process $y$ from the local time 
of a standard Brownian motion $w$. By the second Ray-Knight Theorem \cite{K63}, the processes
$$
(L^t(\tau_{2x}); t\geq0) \ \mbox{and  } (L^{-t}(\tau_{2x}); t\geq 0) 
$$
are independent  $0$-dimensional squared Bessel process both starting at $2x$. So their sum 
$({\cal L}^t(\tau_{2x}); t\geq0)$  is a $0$-dimensional squared Bessel process starting at $4x$, i.e, $({\cal L}^t(\tau_{2x}); t\geq0)$ has the same distribution as the process $y$. 
Finally, since the hitting time at $0$ of $({\cal L}^t(\tau_{2x}); t\geq0)$ coincides with the maximum of $|w|$ on $[0,\tau_{2x}]$,
this shows that the LHS of (\ref{jhf}) is identical in law to $ \max_{s\in[0,\tau_{2x}]} |w(s)|$.\\

We now show that the same identity holds for the RHS of  (\ref{jhf}).
For every $l\in\R^+$ such that 
$\tau_{l-}<\tau_{l}$, define $m_l$ 
the height of the excursion on the interval $[\tau_{l-}, \tau_l]$ 
(i.e, $m_l \ := \  \sup\{|w(s)| \ : \ s\in[\tau_{l-}, \tau_l] \}$). From standard excursion theory (see e.g., Section VI.47 \cite{RW87}), 
$\sum_{l \ : \ \tau_{l-}<\tau_{l} } \delta_{l,m_l}$
defines 
a Poisson Point Process with intensity measure $dl\otimes \frac{dt}{t^2}$. It follows that 
for every $a$:
$$
\max\{m_l \ :  \ \tau_{l-}<\tau_l \mbox{  and  } l\leq a\} \ =  \ \max_{s\in [0,\tau_a]} |w(s)|.
$$
Finally since the intensity measure underlying ${\cal C}$ is twice the intensity measure of $\sum_{l \ : \ \tau_{l-}<\tau_{l} } \delta_{l,m_l}$, 
using standard results 
on Poisson point processes, it is not hard to show that the RHS of (\ref{jhf})
is distributed as $\max_{s\in[0,\tau_{2x}]} |w(s)|$. This  completes the proof of Lemma \ref{id}. 
\end{proof}

\section{Conditional Sampling and Sized-Biased Killed Brownian CPP}
\label{c-s}


Let $C$ be the Kingman comb as defined in Proposition \ref{prop:Kingman-comb}. 
For every $\eps>0$, let us consider the partition ${\cal P}_{\eps}$ of $[0,1]$
induced by the equivalence relation
$$
x\sim_{(\eps)} y  \Longleftrightarrow \ \phi_{-\eps,0}(x)=\phi_{-\eps,0}(y).
$$
We let $N_{\eps}$ be the number of equivalence classes (or blocks).  To 
characterize those equivalence classes, let $\tilde C$ be the comb that coincides with $C$ on $(0,1)$ and with $\tilde C(0)=\tilde C(1)=1$
and define the sequence $(L_{\eps}^i)_{i=0}^{N_\eps+1}$ 
where $(L_{\eps}^i)$
is the ranked enumeration of the set $\{x \ : \ \tilde C(x)>\eps\}$. 
It is straightforward to check that the $i^{th}$ block (where blocks are ranked according to their least element) coincides with the interval $(L^{i-1}_\eps, \ L^i_\eps )$. 
For $i\leq N_\eps$, we define the comb $\Bi_{\eps}^i$ as
$$
\Bi_{\eps}^i \ := \ k_{ l^{i}_\eps }C( \cdot + L^{i-1}_\eps  ) \mbox{    with $l_\eps^{i} := L^i_\eps - L^{i-1}_\eps$}
$$
that can be  thought of as the comb encoding the genealogy of the $i^{th}$ block.

This motivates the definition of the comb $M_{n,\eps}$ whose law is characterized as follows. For every bounded measurable $f: \Omega \times(0,\infty)\to\R$,
\begin{equation}\label{eq:def-cond-sampling}
\E\left(f(M_{n,\eps}, l_{n,\eps}) \right) \ := \ \E\left(\frac{\sum_{i=1}^{N_{\eps}} (l_{\eps}^i)^n\, f(\Bi_{\eps}^i, l_{\eps}^i)}{\sum_{i=1}^{N_{\eps}} (l_{\eps}^i)^n}\right).
\end{equation}
In words,
conditioned on a realization of the population, we first generate $n$ independent uniform random variables 
on $[0,1]$, independent of the flow of bridges. If we condition those random variables to fall into the same block,
$l_{n,\eps}$ is defined 
as the size of the chosen block and
$M_{n,\eps}$ is the comb encoding the genealogy of the (averaged) block. The aim of this section is to prove the following result.

\begin{theorem}\label{cvc}
Let  $\Mi$ be a CPP  with intensity measure $\nu(dl) 1_{l<1}$ and ${\cal L}_n$ be an independent Gamma distributed random variable with parameter $(n+1,2)$. Then the following convergence
$$
\left(S_\eps ( M_{n,\eps}), \frac{1}{\eps} l_{n,\eps} \right) \ \Longrightarrow \ (k_{{\cL}_n}(\Mi), {\cal L}_n)
$$
holds weakly in law as $\eps\to 0$. \end{theorem}

\begin{remark}\label{rem-cr}
As already mentioned in the introduction, the (properly rescaled) coalescent point process associated to a critical birth-death process
conditioned on surviving up to time $N$ converges to $k_{\cL_0}(\Mi)$ as $N\to\infty$ \cite{P04}. As a consequence, the previous result states that after rescaling, the limit of $(M_{n,\eps})$ is described in terms a `$n$-size-biased critical birth-death process' (i.e., biased by the $n$-th power of its size) conditioned on  survival up to a large time.
\end{remark}

\subsection{The $n^{th}$ moment of the length of a uniformly chosen block}\label{sect:cv-l2}
In order to prove Theorem \ref{cvc}, we will need to first establish some technical results.
Define
$$
X_{n,\eps} \ = \  \ \frac{1}{N_\eps} \ \sum_{i=1}^{N_{\eps}} \left(\frac{2}{\eps}l_i^\eps\right)^n.
$$
The aim of this section is to prove the following proposition.

\begin{proposition}\label{UI}
As $\eps\to 0$, the sequence $(X_{n,\eps})$ converges to $n !$ in $L^2$. 
\end{proposition}
Before proceeding with the proof, we first deduce an easy corollary of Proposition \ref{UI}.

\begin{corollary}\label{UII}
The collection of r.v. $\{(\frac{2}\eps l_1^\eps)^n\}_{\eps>0}$ is uniformly integrable.
\end{corollary}
\begin{proof}
Let $V_\eps = (\frac{2}\eps l_1^\eps)^n$. Applying Cauchy--Schwartz inequality:
\begin{eqnarray*}
\E(V_\eps 1_{V_\eps>M})^2 & \le & \E(V_\eps^2) \P(V_\eps>M) \\
					   & =  & \E\left( \ \E\left( \left.\left(\frac{2}\eps l_1^\eps\right)^{2n} \ \right| \ N_\eps\right) \ \right)  \ \P(V_\eps>M)  \\
					   & = &  \E\left(  \ \E\left(\left.\frac{1}{N_\eps}\sum_{i=1}^{N_\eps}\left(\frac{2}\eps l_i^\eps\right)^{2n} \ \right|\ N_\eps\right)  \ \right )  \ \P(V_\eps>M) \\
                                              & = & \E(X_{2n,\eps}) \P(V_\eps>M), 					    	 
\end{eqnarray*}
where the second equality follows from the fact that conditioned on $N_\eps$,
the random variables $\{l_i^\eps\}_{i=1}^{N_\eps}$ are exchangeable. On the the one hand, Proposition \ref{UI}
implies that 
$$
\sup_{\eps>0} \E(X_{2n,\eps}) <\infty.
$$
On the other hand, Proposition \ref{cvc2} implies that  $V_\eps$
converges in law to $(l_1({\mathcal C}))^n$, and thus $\{V_\eps\}_{\eps>0}$ is tight.
This ends the proof of the lemma.
\end{proof}
%
%
We now proceed with the proof of Proposition \ref{UI}. We first need to define the r.v. $Z_{n,\eps}$.
Let us consider a sequence $\{e_k\}_{k\geq 1}$ of independent random variables where $e_k$ is an exponential random variable with parameter 
$\frac{k(k-1)}{2}$ and an independent sequence $\{\xi_k\}_{k\geq 1}$ of i.i.d exponential random variables with parameter $1$. Now set
\begin{equation}\label{zeps}
Z_{n,\eps} := \left( \ \frac{1}{\bar N_\eps} \ \sum_{k=1}^{\bar N_\eps} \ \left(\frac{\xi_k}{S_{\bar N_\eps}} \frac{2}{\eps}\right)^n \  \right),
\end{equation}
where   $\bar N_\eps := \ \inf\{n \ : \ \sum_{k\geq n} e_k < \eps\}$, and $S_n=\sum_{k=1}^n \xi_k$.

\begin{lemma}\label{lemma:dist}
For all $n$ and $\eps$, the r.v. $Z_{n,\eps}$ and $X_{n,\eps}$ are equally distributed.
\end{lemma}
\begin{proof}
For the standard Fleming-Viot flow of bridges, the number of blocks 
$N_\eps$  is distributed as the value at time $\eps$
of a pure death process descending from $\infty$
with rate $k(k-1)/2$ at level $k$, and further, conditioned on $N_\eps$,
the vector $\{l_i^\eps\}_{i=1}^{N_\eps}$
is distributed as a $N_\eps$-dimensional Dirichlet random variable with parameter 
 $(1,\cdots,1)$. Finally, it is well known that
 the n-dimensional Dirichlet random variable with parameter 
 $(1,\cdots,1)$ is distributed as  
 $$
 (\xi_1/S_n,\cdots,\xi_n/S_n),
 $$
which completes the proof of the lemma.
\end{proof}

\begin{lemma}\label{neps}
Let $\gamma>0$ and define
\begin{eqnarray*}
N_{\eps,\gamma}^- \ = \ \left\lfloor(1-\gamma)\frac{2}{\eps}\right\rfloor, \ \mbox{and} \ \ N_{\eps,\gamma}^+ \ = \ \left\lceil(1+\gamma)\frac{2}{\eps}\right\rceil.
\end{eqnarray*}
Then for any $k\in\N$, $\P\left( \ \bar N_\eps \notin[N_{\eps,\gamma}^-, N_{\eps,\gamma}^+] \ \right)/\eps^k\to0$
as $\eps\to0$.
\end{lemma}

\begin{proof}
If  $R_N=\sum_{k\geq N} e_k$, then
$$
\E(\exp(-\lambda N^2 R_N)) \ =  \ \exp\left(\sum_{k\geq N} \ \ \ \log\left( \frac{\frac{k}{2N}(\frac{k}{N} - \frac{1}{N})}{ \lambda  \ + \ \frac{k}{2N}(\frac{k}{N} - \frac 1 N)  } \right) \ \right) 
$$
For every $\lambda>-1/2$,
the function $u\to \ln\left(\frac{\frac{u^2}{2}}{\frac{u^2}{2}+\lambda}\right)$ is integrable on $[1,\infty)$. Thus, as $N\to \infty$,
it is not hard to show the following convergence
\begin{equation}
\label{eqr1}
\frac{1}{N}\sum_{k\geq N} \log\left( \frac{\frac{k}{2N}(\frac{k}{N}-\frac 1 N)}{ \lambda  \ + \ \frac{k}{2N}(\frac k N - \frac 1 N)} \right) \to \int_1^\infty \ln\left(\frac{\frac{u^2}{2}}{\frac{u^2}{2}+\lambda}\right) du,
\end{equation}
or equivalently,
\begin{equation}\label{eq:exp-mom}
\E( \ \exp(-\lambda N^2 R_N) \ ) \ = \ \exp\left( \ -N \ \int_1^\infty \ln\left(\frac{\frac{u^2}{2}}{\frac{u^2}{2}+\lambda}\right) du \ \right) \gamma_N(\lambda).
\end{equation}
with $\gamma_N(\lambda)\to 1$ as $N\to\infty$. From Chebyshev's inequality this implies:
\begin{eqnarray*}
\P\left(\  N R_N  < 2(1-\gamma) \right) 
& \leq &  \exp\left(\lambda N 2(1-\gamma)\right) \ \E(\exp(-\lambda N^2 R_N)) \\
&   =   & \exp\left(N(\lambda 2(1-\gamma) - \ f(\lambda) )\ \right) \ \gamma_N(\lambda) \\
\mbox{with}
&  &
f(\lambda) = \int_1^\infty \ln\left(\frac{\frac{u^2}{2}}{\frac{u^2}{2}+\lambda}\right) du.
\end{eqnarray*}
Since $f(0)=0$ and $f'(0)=2$, we can choose $\lambda>0$ small enough such that 
$\lambda(2-\gamma) - \ f(\lambda)<0$. By the same token,
for every $\lambda\in(-1/2,0)$ and $N\in\N$:
\begin{eqnarray*}
\P\left(\  N R_N  \geq 2(1+\gamma) \right) 
& < &  \exp\left(\lambda N 2(1+\gamma)\right) \ \E(\exp(-\lambda N^2 R_N)) \\
&   =   & \exp\left(N(\lambda 2(1+\gamma) - \ f(\lambda) )\ \right) \ \gamma_N(\lambda)
\end{eqnarray*}
and again, we can choose $\lambda<0$ small enough such that 
$\lambda(2+\gamma) - \ f(\lambda)<0$. To complete the proof, we  combine the previous large deviation estimates with the observation
$$
\{\bar N_\eps < N_{\eps,\gamma}^+\} \ = \ \{ R_{N_{\eps,\gamma}^+} \geq \eps\} \ = \ \{ N_{\eps,\gamma}^+  R_{N_{\eps,\gamma}^+} \geq N_{\eps,\gamma}^+ \eps\}
$$
with $N_{\eps,\gamma}^+ \eps = 2(1+\gamma)$ and
$$
\{\bar N_\eps \geq N_{\eps,\gamma}^-\} \ = \ \{N_{\eps,\gamma}^- R_{N_{\eps,\gamma}^-} \leq N_{\eps,\gamma}^- \eps\}
$$
with $N_{\eps,\gamma}^- \eps = 2(1- \gamma)$. Combining this with the previous inequalities shows that
$\P\left( \ \bar N_\eps \notin[N_{\eps,\gamma}^-, N_{\eps,\gamma}^+] \ \right)$ goes exponentially fast to $0$ in $\eps$ as $\eps\to0$. This  completes the proof of Lemma 
\ref{neps}.
\end{proof}

\begin{proof}[Proof of Proposition \ref{UI}] From Lemma \ref{lemma:dist}, 
it is enough to show that $Z_{n,\eps}$ converges to $n!$ in $L^2$. To ease the notation, we drop the dependence in $n$ and write $Z_\eps\equiv Z_{n,\eps}$. Let us now introduce two auxiliary variables:
\begin{eqnarray*}
Z_\eps^+ = \left( \ \frac{1}{N_{\eps,\gamma}^-} \ \sum_{k=1}^{N_{\eps,\gamma}^+} \ \left(\frac{\xi_k}{S_{N_{\eps,\gamma}^-}} \frac{2}{\eps}\right)^n \  \right), \ \ \ \ 
Z_\eps^- = \left( \ \frac{1}{N_{\eps,\gamma}^+} \ \sum_{k=1}^{N_{\eps,\gamma}^-} \ \left(\frac{\xi_k}{S_{N_{\eps,\gamma}^+}} \frac{2}{\eps}\right)^n \  \right) 
\end{eqnarray*}
Since $Z_\eps\ <  \ \left(\frac{2}{\eps}\right)^n$ and $Z_{\eps}^-\leq Z_\eps \leq Z_\eps^+$ on $\{\bar N_\eps\in[N_{\eps,\gamma}^-, N_{\eps,\gamma}^+]\}$, for every $\gamma>0$:
$$
\E\left(Z_\eps-n!\right)^2 < \left(n! +  \left(\frac{2}{\eps}\right)^n\right)^2 \P(\bar N_\eps\notin[N_{\eps,\gamma}^-, N_{\eps,\gamma}^+]) \ + \ \E(Z_{\eps,\gamma}^+-n!)^2 \ + \   \E(Z_{\eps,\gamma}^--n!)^2. 
$$
Lemma \ref{neps} implies that the first term vanishes as $\eps\to0$ and it remains to show that 
$$
\lim_{\gamma\to0} \ \lim_{\eps\to0} \  \E(Z_{\eps,\gamma}^+-n!)^2, \  \E(Z_{\eps,\gamma}^--n!)^2 =0
$$
We only show the first limit. The second limit can be shown along the same lines.
We can rewrite
\begin{eqnarray*}
Z_{\eps,\gamma}^+ 
& = & 
\frac{N_{\eps,\gamma}^+}{N_{\eps,\gamma}^-}(1-\gamma)^n \ \frac{1}{N_{\eps,\gamma}^+} 
\sum_{k=1}^{N_{\eps,\gamma}^+} \ \left(\frac{\xi_k}{S_{N_{\eps,\gamma}^-}} \frac{2}{\eps(1-\gamma)}\right)^n \\
& = & \ c_{\eps,\gamma}  \  \ 
\underbrace{\frac{1}{N_{\eps,\gamma}^+} \sum_{k=1}^{N_{\eps,\gamma}^+} \ ({\xi_k})^n}_{:=A_{N_{\eps,\gamma}^+}}  \ \  
\underbrace{\left(\frac{N^-_{\eps,\gamma}}{S_{N^-_{\eps,\gamma}}}\right)^n}_{:=B_{N_{\eps,\gamma}^-}} 
\end{eqnarray*}
with 
$$
c_{\eps,\gamma}= \frac{N_{\eps,\gamma}^+}{N_{\eps,\gamma}^-}(1-\gamma)^n \left(\frac{2 N_{\eps,\gamma}^-}{\eps}\right)^n \ \ \mbox{so that 
$\lim_{\gamma\to0} \ \lim_{\eps\to0} c_{\eps,\gamma}=1$.}
$$
Next, writing $A=\E(\xi^n)=n!$ and $B=\left(\frac{1}{\E(\xi)}\right)^n=1$, we have
\begin{eqnarray*}
 A_{N_{\eps,\gamma}^+}B_{N_{\eps,\gamma}^-} - AB & = & B  (A_{N_{\eps,\gamma}^+} - A) \ + \ A(B_{N_{\eps,\gamma}^-}-B) + (A_{N_{\eps,\gamma}^+}-A)(B_{N_{\eps,\gamma}^-}-B) \\
							      & = &  B  (A_{N_{\eps,\gamma}^+} - A) \ + \ A_{N_{\eps,\gamma}^+} B_{N_{\eps,\gamma}^-}(1-\frac{B}{B_{N_{\eps,\gamma}^-}}) \ + \ B_{N_{\eps,\gamma}^-}(  A_{N_{\eps,\gamma}^+}-A)(1-\frac{B}{B_{N_{\eps,\gamma}^-}}). 
\end{eqnarray*}
and in order to prove Proposition \ref{UI}, it remains to show that the RHS of the equality goes to $0$ in $L^2$ as $\eps$ goes to $0$.
By applying Cauchy--Schwartz inequality several times, it is easy to show that the RHS 
converges to $0$ in $L^2$ if for all $p\in\N$,
\begin{eqnarray}
 & \ (A_N -  A)  \to 0 \mbox{   and} \ \ (\frac{B}{B_{N}}-1) \to 0 \ \ \mbox{as $N\to\infty$ in $L^p$}  \label{lln}\\
 \mbox{and} &  \ \sup_{N} \ \E\left(A_{N} ^p\right)<\infty  \ \ \mbox{and} \ \ \sup_{N} \ \E\left(B_N^p\right)<\infty. \label{sup}
\end{eqnarray}
Condition (\ref{lln}) is a law of large number and can easily be checked using the fact that the $\xi_i$'s are i.i.d exponential random variables.
The first condition of (\ref{sup}) directly follows from the first part of (\ref{lln}). For the second assertion of (\ref{sup}),
we first note that $S_N$ is distributed as a Gamma distribution with parameter $(N,1)$, and thus:
\beqnn
\E(B_N^p) \ = \ N^{np} \frac{1}{\Gamma(N)} \int_{0}^\infty x^{N-1-np} e^{-x} dx \ = \ N^{np} \frac{(N-np-1) !}{(N-1)!} \ \to 1
\eeqnn
as $N\to\infty$.
This ends the proof of Proposition \ref{UI}.
\end{proof}

\subsection{Proof of Theorem \ref{cvc}}

We now proceed with the proof of Theorem \ref{cvc}.
Let $0<x_0<x_1<\cdots <x_{n-1}$ and let
$
g:\R^n\rightarrow \R
$
be an arbitrary bounded and continuous function. Define $f:\Omega \times (0,\infty)\to\R$ as  
$$
f(\omega, x) \ := \ g\left(d_\omega(x_0,x_1), \cdots, d_\omega(x_{n-2},x_{n-1}),  x\right).
$$
We aim at showing that
\beqn\label{sdf22}
\lim_{\eps\to 0}\E\left(f(S_\eps(M_{n,\eps}), l_{n,\eps}/\eps)\right) \ & =&  \E\left(f(k_{\cL_n}({\cal M}), \cL_n) \right), 
\eeqn
where $({\cal M}, \cL_n)$ are defined as in Theorem \ref{cvc}. Set  $S_\eps(\Bi_{\eps}^i, l_{\eps}^i)= (S_\eps(\Bi_{\eps}^i), \frac{1}{\eps}l_{\eps}^i)$ and
\begin{eqnarray*}
x_\eps =  \sum_{i=1}^{N_{\eps}} f\circ S_\eps(\Bi_{\eps}^i, l_{\eps}^i)(l_i^\eps)^n \ / \ \sum_{i=1}^{N_{\eps}} (l_i^\eps)^n  \ \ \ , \ \ 
y_\eps =  \frac{1}{n!} \frac{1}{N_{\eps}}\sum_{i=1}^{N_{\eps}} f\circ S_\eps(\Bi_{\eps}^i, l_{\eps}^i)  (\frac{2}{\eps} l_{\eps}^i)^{n} 
\end{eqnarray*}
Straightforward manipulation yields that for all $\eps >0$,
$$
| x_\eps-y_\eps  | 
\ \leq \  ||g||_\infty \left| 1 - \frac{X_{n,\eps}}{n!} \right|.
$$
where $X_{n,\eps}$ is defined as in Section \ref{sect:cv-l2}.
From Proposition \ref{UI}, the RHS goes to $0$ in $L^1$ 
as $\eps\to0$. Now by definition
$$
\E\left(f(S_\eps(M_{n,\eps}), l_{n,\eps}/\eps)\right) = \E(x_\eps),
$$
and thus
$$
\lim_{\eps\to0}   \ \E\left(f(S_\eps(M_{n,\eps}), l_{n,\eps}/\eps)\right) \ - \ \frac{1}{n !}\E\left( \frac{1}{N_{\eps}}\sum_{i=1}^{N_{\eps}} f\circ S_\eps(\Bi_{\eps}^i, l_{\eps}^i) \left(\frac{2}{\eps} l_{\eps}^i  \right)^n \right)    \ = \ 0.
$$
Next, using exchangeability, we get that
\beqnn
\E\left( \ \frac{1}{N_{\eps}}\sum_{i=1}^{N_{\eps}} f\circ S_\eps(\Bi_{\eps}^i, l_{\eps}^i) \left(\frac{2}{\eps} l_{\eps}^i  \right)^n \right) & = &
\left.\E\left( \ \E\left(\ \frac{1}{N_{\eps}}\sum_{i=1}^{N_{\eps}} f\circ S_\eps(\Bi_{\eps}^i, l_{\eps}^i) \left(\frac{2}{\eps} l_{\eps}^i  \right)^n \ \right| \ N_\eps \right) \ \right) \\
& = &\left.
\E\left( \ \E\left( f\circ S_\eps(\Bi_{\eps}^1, l_\eps^1) \left(\frac{2}{\eps} l_{\eps}^1  \right)^n \ \right| \ N_\eps \right) \ \right) \\
& = &
\E\left( f\circ S_\eps(\Bi_{\eps}^1, l_\eps^1) \left(\frac{2}{\eps} l_{\eps}^1 \right)^n \right). 
\eeqnn
By Proposition \ref{cvc2}, 
$
S_\eps(\Bi_\eps^1,   l_{\eps}^1)
$
converges weakly to
$
(\bar {\cal C}, l_1(\Ci)),
$
where $\bar {\cal C} = k_{l_1}{\cal C}$.
Using the uniform integrability of  $((\frac{2}{\eps}l_1^\eps)^n)_{\eps>0}$ (by Corollary \ref{UII}),
this yields
\beqn\label{sdf2}
\lim_{\eps\to 0}
\E\left(f(S_\eps(M_{n,\eps}), l_{n,\eps}/\eps)\right) \ & = & \ \frac{2^n}{n !}  \E\left(f({\bar \Ci}, l_1(\Ci)) l_1(\Ci)^n \right). 
\eeqn
To complete the proof of (\ref{sdf22}) (and thus of Theorem \ref{cvc}), we first note that 
$l_1({\mathcal C})$  is an exponential random variable with parameter $2$, since 
$$
\P(l_1({\cal C}) \geq x) \ = \ \exp\left(-2\int_{l\in(0,x)}\int_{u\in[1,\infty)}   \ \frac{du \ dl}{u^2} \right) \ = \ \exp(-2 x),
$$
Second, note that $(\bar \Ci, l_1(\Ci))$ is identical in law
to the pair $(\bar \Ci', l')$ were 
$\bar \Ci'$ is obtained by killing at the exponential random variable $l'$
an independent CPP $\cal M$ with intensity measure 
$\nu(dl) 1_{l< 1}$. 
In other words
$$
\lim_{\eps\to 0}
\E\left(f(S_\eps(M_{n,\eps}), l_{n,\eps}/\eps)\right) \ = \ \frac{2^n}{n !} \int_0^{\infty}dx\, 2e^{-2x}\ \E\left(f(k_{x}({\cal M}), x) \right)x^n,
$$
which ends the proof of Theorem \ref{cvc}.

\section{Genealogies Associated with Two Conditional Samplings}

\subsection{Quenched Conditional Sampling}\label{c-s-2}

We now consider the genealogy of the $n$ uniformly sampled individuals after quenched conditional sampling as defined in the previous section (see Equation \eqref{eq:def-cond-sampling}).
For a given realization of ${M}_{n,\eps}$, let $U_0,\ldots,U_{n-1}$
be $n$ independent uniform random variables on the interval $[0,l_{n,\eps}]$.
Let $(U_{(0)},\ldots,U_{(n-1)})$ be the vector $(U_0,\ldots,U_{n-1})$ reordered 
from least to greatest and define
$$
H_i^\eps := d_{M_{n,\eps}}({U_{(i-1)}, U_{(i)}}) \ \ \quad i\in\{1,\ldots,n-1\}
$$
the coalescence times of our sample.
\begin{theorem}\label{teo1}
As $\eps\to 0$, the coalescence times
$
(\eps^{-1}H_i^\eps, \ i\in\{1,\cdots,n-1\}) 
$
converge to $n-1$ i.i.d.  uniformly distributed random variables on $(0,1)$.
\end{theorem}
In order to show Theorem \ref{teo1}, we define  the ${\cal H}_i$'s analogously to $H_i^\eps$, but with respect 
to $k_{\cL_n}({\cal M})$, that is
$$
{\cal H}_i := d_{\cal M}({\cal U}_{(i-1)}, {\cal U}_{(i)}),
$$
where the $\{{\cal U}_{(i)}\}_{i=0}^{n-1}$ form the reordering of $n$ independent uniform random variables on the interval $[0, {\cal L}_{n}]$. We decompose the proof of Theorem \ref{teo1} into two steps. 
We first show that 
$\{H_i^\eps\}$ converges to $\{{\cal H}_i\}$ (see Lemma \ref{cvh}). We then characterize the distribution of the 
 ${\cal H}_i$'s (see Lemma \ref{cor-final}).

\begin{lemma}\label{cvh}
As $\eps\to 0$, we have the convergence in distribution of 
$
(\eps^{-1}H_i^\eps, \ i\in\{1,\ldots,n-1\})$ to $({\cal H}_i, \ i\in\{1,\ldots,n-1\}).
$
\end{lemma}

\begin{proof}
Let $f$ be a bounded continuous function from $\R^{n-1}$ to $\R$
and let $x_0<x_1<\cdots<x_{n-1}$.
Define
\begin{eqnarray*}
g_\eps(x_0, x_1,\ldots,x_{n-1}) & = & \E\left( f\left( d_{S_\eps(M_{n,\eps})}(x_{0},x_1), \ldots, d_{S_\eps(M_{n,\eps})}(x_{n-2}, x_{n-1}) \right) 1_{\eps x_{n-1}< l_{n,\eps}} \frac{\eps^n}{l_{n,\eps}^{n}} \right) \\
g_0(x_0, x_1,\ldots,x_{n-1}) & = &  \E\left( f\left( d_{\Mi}(x_{0},x_1), \ldots, d_{\Mi_{n}}(x_{n-2}, x_{n-1}) \right) 1_{x_{n-1}< \cL_{n}} \frac{1}{\cL_{n}^{n}} \right).
\end{eqnarray*}
It is not hard to see that 
$$
\E\left( f( \eps^{-1}H_1^\eps,\ldots, \eps^{-1}H_{n-1}^\eps  \right) ) \ = \ n !  \int_{\Delta_{n}} g_{\eps}(\bx) d \lambda(\bx) \
$$
where $\lambda$ is the Lebesgue measure on $\R^{n}$ and $\Delta_{n} = \{\bx \in\R^{n} \ : \ x_0 < x_1 <\cdots <x_{n-1}\}$. Further, an analogous relation 
holds for the limiting object. Thus 
we need to show that 
\beqn\label{tre}
\lim_{\eps\rightarrow 0} \int_{\Delta_{n}} g_{\eps}(\bx) d \lambda(\bx) \ = \  \int_{\Delta_{n}} g_{0}(\bx) d \lambda(\bx).
\eeqn
Since $\P(\cL_n=x)=0$ for every $x\in\R$, Theorem \ref{cvc} implies that $g_\eps$ converges almost surely to $g_0$ on $\{0<x_1<\cdots<x_{n-1}\}$, 
and thus, it remains to justify
the limit-integral interchange in (\ref{tre}).
For every $\delta<A\in(0,\infty)$:
\beqn
| \int_{\Delta_{n}}  \left(g_{\eps}(\bx) \ - \ g_{0}(\bx) \right) d \lambda(\bx)|
& < & \int_{\Delta_{n}\cap\{x_{n-1}>A\}} |g_\eps(\bx)| d\lambda(\bx)  +  \int_{\Delta_{n}\cap\{x_{n-1}>A\}}   |g_0(\bx)| d\lambda(\bx)\nonumber  \\
& + &  \int_{\Delta_{n}\cap\{x_{n-1}<\delta\}} |g_\eps(\bx)| d\lambda(\bx)  +  \int_{\Delta_{n}\cap\{x_{n-1}<\delta\}}   |g_0(\bx)| d\lambda(\bx) \nonumber \\
&  +     &  |\int_{\Delta_{n}\cap\{x_{n-1}\in[\delta,A]\}}   \left(g_{\eps}(\bx) \ - \ g_{0}(\bx) \right)  d \lambda(\bx)|. \label{edl}
\eeqn 
The last term goes to $0$ by combining the bounded convergence theorem and Theorem \ref{cvc}. 
Next, let $M$ be such that $||f||_\infty < M$
\beqnn
 \int_{\Delta_{n}\cap\{x_{n-1}>A\}} |g_\eps(\bx)| \ d\lambda(\bx)  
 & < & M  \int_{\Delta_{n}\cap\{x_{n-1}>A\}} \E\left(1_{\eps x_{n-1}<l_{n,\eps}} \frac{\eps^n}{l_{n,\eps}^{n}}\right)\  d\lambda(\bx) \\
 & =    &  M  \E\left( \int_{\Delta_{n}} 1_{A<x_{n-1}<l_{n,\eps}/\eps} \frac{\eps^n}{l_{n,\eps}^{n}}\  d\lambda(\bx) \right) \\
 & =  & \frac{M}{n !}  \E( 1_{l_{n,\eps/\eps}>A}\left(1-\left(\frac{\eps A}{l_{n,\eps}}\right)^{n}\right)  \ 
  <  \ \frac{M}{n !} \P(l_{n,\eps}/\eps >A)  
\eeqnn
By a similar computation, we get 
\beqnn
 \int_{\Delta_{n}\cap\{x_{n-1}>A\}} |g_0(\bx)| \ d\lambda(\bx)  
 < \frac{M}{n!} \P(\cL_n >A)  
\eeqnn
Next, we control the second line of (\ref{edl}).
\beqnn
 \int_{\Delta_{n}\cap\{x_{n-1}<\delta\}} |g_\eps(\bx)| d\lambda(\bx)  
& \leq & M  \E\left( \int_{\Delta_{n}\cap\{x_{n-1}<\delta\}} 1_{x_{n-1} < l_{n,\eps} /\eps} \frac{\eps^n}{l_{n,\eps}^{n}}\  d\lambda(\bx), \ l_{n,\eps}/\eps< \delta\right) \\ 
& + & \ M\E\left( \int_{\Delta_{n}\cap\{x_{n-1}<\delta\}} \frac{\eps^n}{l_{n,\eps}^{n}} \ d\lambda(\bx), \ l_{n,\eps}/\eps>\delta  \right) \\
& = & \frac{M}{n !} \left(  \P(l_{n,\eps}/\eps< \delta) \ + \ \delta^{n} \E\left(\frac{\eps^n}{l_{n,\eps}^{n}}, \ l_{n,\eps}/\eps >\delta\right) \right)  .
\eeqnn
 By a similar computation, we get 
 $$
 \int_{\Delta_{n}\cap\{x_{n-1}<\delta\}} |g_0(\bx)| d\lambda(\bx)  \leq  \frac{M}{n !} \left(  \P(\cL_n < \delta) \ + \ \delta^{n} \E(\frac{1}{\cL_n^{n}}, \ \cL_n >\delta) \right).  
 $$
Collecting the previous bounds, using the fact that $l_{n,\eps}/\eps\Longrightarrow \cL_n$ and the bounded convergence theorem
 yield that for every $\delta<A\in(0,\infty)$
\beqn
\limsup_{\eps\downarrow 0} \ \left| \int_{\Delta_{n}}  \left(g_{\eps}(\bx) \ - \ g_{0}(\bx) \right) d \lambda(\bx)\right| 
& \leq &
\frac{2M}{n!}\left( \P(\cL_n\geq A) \ + \ \P(\cL_n< \delta) \ + \ \delta^{n} \E(\frac{1}{\cL_n^{n}}, \ \cL_n>\delta) \right) \nonumber \\
\label{eq:last}
\eeqn
Since $\cL_n$ is distributed as a Gamma random variable 
with parameter $(n+1,2)$,
we have
\begin{eqnarray*}
\limsup_{\delta\to0} \ \delta^{n} \E(\frac{1}{\cL_n^{n}}, \ \cL_n>\delta) 
& = & \frac{2^n}{n !}\limsup_{\delta\to 0} \delta^n \int_{\delta}^{\infty} e^{-2x} dx \\
& = &  0.
\end{eqnarray*}
Taking the limit $\delta\rightarrow0$ and $A\rightarrow\infty$ in (\ref{eq:last}) yields the desired result.
\end{proof}

We now characterize the distribution of the vector $({\cal H}_i, i\in\{1,\ldots,n-1))$.

\begin{lemma}\label{cor-final}
The r.v. $({\cal H}_i, i\in\{1,\ldots,n-1\})$ are i.i.d. uniformly distributed on $[0,1]$. 
\end{lemma}

\begin{proof}
Recall that ${\cal L}_{n}$ is a Gamma r.v. with parameter $(n+1, 2)$ and that conditional on ${\cal L}_n$, the $\{{\cal U}_{(i)}\}_{i=0}^{n-1}$ form the reordering of $n$ independent uniform random variables $\{{\cal U}_{i}\}_{i=0}^{n-1}$ on the interval $[0, {\cal L}_{n}]$. So the $(n+1)$-tuple $({\cal U}_{0},\ldots,{\cal U}_{n-1}, {\cal L}_n)$ has measure on $\R^{n+1}$ 
$$
\frac{2^{n+1}}{n !}\exp(-2 L) L^n dL \ \times \ \frac{du_0}{L} 1_{0<u_0<L}\  \times \cdots \times \frac{d u_{n-1}}{L} 1_{0<u_{n-1}<L},
$$
that is, $({\cal U}_{0},\ldots,{\cal U}_{n-1}, {\cal L}_n)$ has a density  equal to $
\frac{2^{n+1}}{n !}e^{-2 L}$ on the following subset of $\R_+^{n+1}$ 
$$
A:=\{(u_0,\ldots, u_{n-1}, L): \forall i,\ 0<u_i<L\}.
$$
Now there are $n!$ pairwise disjoint $\{A_i\}_{i=1}^{n!}$ open subsets of $A$ corresponding to the $n!$ possible orderings of  $(u_0,\ldots, u_{n-1})$ and whose union coincides with $A$ $\lambda$-a.e. In addition, on each $A_i$, the map  
\beqnn
A_i & \rightarrow & \R_+^{n+1} \\
(u_0, \ldots, u_{n-1}, L)              & \mapsto & (u_{(0)}, u_{(1)}- u_{(0)},\ldots, u_{(n-1)}- u_{(n-2)}, L- u_{(n-1)})
\eeqnn
is a $C^1$-diffeomorphism with jacobian equal to 1 in absolute value. As a consequence, the $(n+1)$-tuple $({\cal U}_{(0)},{\cal U}_{(1)}-{\cal U}_{(0)},\ldots,{\cal U}_{(n-1)}-{\cal U}_{(n-2)}, {\cal U}_{(n)}- {\cal U}_{(n-1)})$, where we have set ${\cal U}_{(n)}:={\cal L}_n$ has a density on $\R_+^{n+1}$ equal to 
$$
2^{n+1} e^{-2 L} = \prod_{i=0}^{n} 2\,e^{-2(u_{(i)}-i_{(i-1)})},
$$
(denoting $u_{(-1)}=0$), which proves that these r.v. are i.i.d. exponential with parameter 2.
Now observe that for any $\sigma\in [0,1]$ and $0<a<b$,
$$
\P \left( \sup\{\Mi(x), x\in [a,b] \} \  \leq \ \sigma  \right) \ = \  
\exp\left(-(b-a) \int_{\sigma}^1 \frac{2ds}{s^2}\right)  \ = \  
\exp\left(-2(b-a) \left(\sigma^{-1}-1\right)\right) . 
$$
From the fact that $\Mi$ is a coalescent point process, for every $(n-1)$-tuple $\{\sigma_i\}_{i=1}^{n-1}$ of $[0,1]$,
\beqnn
 \P\left( \sup\{\Mi(x), x\in {[{\cal U}_{(i-1)},{\cal U}_{(i)}]} \} \  \leq \ \sigma_i ,\ \forall  i\right) & = & \E \left(\prod_{i=1}^{n-1}\exp\left(-2({\cal U}_{(i)}-{\cal U}_{(i-1)}) \left(\sigma_i^{-1}-1\right)\right)\right) \\
 &=& \prod_{i=1}^{n-1}\int_0^\infty dx\,2e^{-2x}\exp\left(-2x \left(\sigma_i^{-1}-1\right)\right)\\
 & = &  \prod_{i=1}^{n-1}\sigma_i, 
\eeqnn
where we have used that $\{{\cal U}_{(i)}-{\cal U}_{(i-1)}\}_{i=1}^{n-1}$ is a vector of i.i.d. exponential random variables with parameter $2$.
\end{proof}

\subsection{Averaged Conditional Sampling}
\label{c-s-3}

We now consider the simpler conditional sampling scheme, namely we consider a uniform sample of $n$ individuals and condition its genealogy to coalesce at a depth smaller than $\eps$. In other words, we consider $T_n < T_{n-1}<\cdots < T_2$ the successive jump times in the $n$-Kingman coalescent and we define $(T_i^\eps, \ i\in\{2,\cdots,n\})$ as the $(n-1)$-tuple $(\eps^{-1}T_i, \ i\in\{2,\cdots,n\})$ conditional on the event $\{T_2<\eps\}$. 
 \begin{theorem}\label{petit-calcul}
 As $\eps\to 0$, the ranked coalescence times 
$
(T_i^\eps, \ i\in\{2,\cdots,n\}) 
$
converge to the order statistics of $n-1$ i.i.d.  uniformly distributed random variables on $(0,1)$.
\end{theorem}
\begin{proof}
For any $i\ge 2$ set $a_i:=i(i-1)/2$ and let $0<t_n<t_{n-1}<\cdots < t_2 <1$. Then setting $T_{n+1}:=0$, we get
\begin{eqnarray*}
\P \left(\eps^{-1}T_i \in dt_i, 2\le i \le n\right) &=& \P \left(\eps^{-1}(T_i-T_{i+1}) \in dt_i-t_{i+1}, 2\le i \le n\right)\\
	&=& \prod_{i=2}^n \P \left(T_i-T_{i+1} \in \eps dt_i-\eps t_{i+1}\right)\\
	&=& \prod_{i=2}^n a_i \eps \, e^{-a_i\eps(t_i-t_{i+1})}dt_i.
\end{eqnarray*}
Integrating this last equation, it is not hard to show that as $\eps \to 0$,
$
\P \left(T_2<\eps\right) \sim \eps^{n-1}\left(\prod_{i=2}^n a_i\right) /(n-1)! 
$, a result that can also be obtained by taking the Laplace transform of $T_2$ and by applying a Tauberian theorem.
This shows that as $\eps\to 0$,
\begin{eqnarray*}
\P\left(T_i^\eps \in dt_i, 2\le i \le n\right) &=& \P \left(\eps^{-1}T_i \in dt_i, 2\le i \le n\vert T_2 <\eps\right) \\
	&\sim& \frac{\prod_{i=2}^n a_i\eps  \, e^{-a_i\eps(t_i-t_{i+1})}dt_i}{ \eps^{n-1}\prod_{i=2}^n a_i /(n-1)!}\\
	&\longrightarrow& (n-1)! \prod_{i=2}^n \,dt_i ,
\end{eqnarray*}
which terminates the proof.
\end{proof}

\paragraph{Acknowledgements.} The authors thank the \emph{Center for Interdisciplinary Research in Biology} (CIRB, Coll\`ege de France) for funding.

\end{document}